\documentclass[12 pt]{amsart}

\usepackage{amssymb, enumitem, mathtools}

\makeatletter
\@namedef{subjclassname@2020}{
  \textup{2020} Mathematics Subject Classification}
\makeatother

\usepackage[T1]{fontenc}

\newtheorem{theorem}{Theorem}[section]
\newtheorem{proposition}[theorem]{Proposition}
\newtheorem{corollary}[theorem]{Corollary}
\newtheorem{lemma}[theorem]{Lemma}

\theoremstyle{definition}
\newtheorem{definition}[theorem]{Definition}
\newtheorem{remark}[theorem]{Remark}

\numberwithin{equation}{section}

\newcommand{\conjC}{\mathbb{C}}
\newcommand{\conjR}{\mathbb{R}}
\newcommand{\conjK}{\mathbb{K}}
\newcommand{\conjT}{\mathbb{T}}
\newcommand{\conjN}{\mathbb{N}}

\newcommand{\cconv}{\overline{\conv}\,}
\DeclareMathOperator{\re}{Re }
\DeclareMathOperator{\conv}{conv }
\DeclareMathOperator{\id}{Id}

\begin{document}

\baselineskip=17pt

\title[]{Polynomial numerical index with respect to a norm-one polynomial}

\author[F. J. Bertoloto]{F\'abio J. Bertoloto}
\address{Institute of Mathematics and Statistics\\ Federal University of Uberl\^andia\\
38.400-902\\ Uberl\^andia -- MG, Brazil}
\email{bertoloto@ufu.br}

\author[E. R. Santos]{Elisa R. Santos}
\address{Institute of Mathematics and Statistics\\ Federal University of Uberl\^andia\\
38.400-902\\ Uberl\^andia -- MG, Brazil}
\email{elisars@ufu.br}

\date{}

\begin{abstract}
In this paper, we introduce the \textit{polynomial numerical index of a pair of Banach spaces with respect to a norm-one polynomial}. This index generalizes the concept of polynomial numerical index defined by Y. S. Choi et al. in 2006 and extends to polynomials the notion of numerical index with respect to an operator, introduced by M. A. Ardalani in 2014. We investigate several properties of this index and present examples.
\end{abstract}



\maketitle

\section{Introduction and preliminaries} 

Let $X$ and $Y$ be Banach spaces over $\conjK$, where $\conjK$ is either $\conjR$ or $\conjC$. We denote by $X^*$ the topological dual of $X$, by $B_X$ the closed unit ball of $X$, and by $S_X$ the unit sphere of $X$. The Banach space of all bounded linear operators from $X$ into $Y$ is denoted by $\mathcal{L}(X,Y)$. In the case where $X = Y$, we simply write $\mathcal{L}(X)$ for $\mathcal{L}(X,X)$.

Given a subset $C$ of $X$, we denote by $\conv C$ or $\conv(C)$ the \textit{convex hull} of $C$, and by $\cconv C$ or $\cconv (C)$ the \textit{closed convex hull} of $C$. We say that $C \subseteq B_{X^*}$ is \textit{norming} for $X$ if $\|x\| = \sup_{\varphi \in C} \re \varphi(x)$ for every $x \in X$, or equivalently, if $B_{X^*} = \cconv^{w^*}(C)$. 

If $k$ is a positive integer, a mapping $P: X \rightarrow Y$ is called a \textit{$k$-homogeneous polynomial} if there exists a $k$-linear mapping $A: X^k \rightarrow Y$ such that $P(x) = A(x, \dots, x)$ for every $x \in X$. We denote by $\mathcal{P}(^k X, Y)$ the vector space of all continuous $k$-homogeneous polynomials from $X$ into $Y$, and by $\mathcal{P}(^k X)$ the space of all continuous scalar-valued $k$-homogeneous polynomials. These spaces are Banach spaces equipped with the norm
$$\|P\| = \sup_{x \in B_X} \|P(x)\|.$$
For background on polynomials on Banach spaces, the reader is referred to \cite{D99, M10}.

Given a non-empty set $\Gamma$, we write $\ell_\infty(\Gamma, Y)$ to denote the Banach space of all bounded functions from $\Gamma$ into $Y$, endowed with the supremum norm.

The concept of \textit{numerical range} on a Hilbert space was first introduced by O. Toeplitz \cite{T18} in 1918 for matrices. In 1961, G. Lumer \cite{L61} extended the theory of the numerical range to bounded linear operators on semi-inner-product spaces. Under the name \textit{field of values}, F. L. Bauer \cite{B62} generalized  the concept to linear operators on finite-dimensional normed linear spaces. F. F. Bonsall et al. \cite{BCS68} further extended the classical notion of numerical range to continuous (not necessarily linear) mappings, thereby broadening its applicability beyond linear operators. Motivated by Lumer's work, L. Harris \cite{H71} generalized the numerical range to holomorphic functions mapping the open unit ball of a Banach space into itself.

In 1968, the concept of numerical index of a Banach space was proposed by Lumer during a lecture to the North British Functional Analysis Seminar. However, it was only in 1970 that J. Duncan et al. \cite{DMPW70} presented a thorough study of the \textit{numerical index of a Banach space}, defined as
\[
n(X) = \inf\{v(T) : T \in \mathcal{L}(X,X), \|T\| = 1\},
\]
where 
$
v(T) = \sup\{|x^*(T(x))| : x \in S_X, x^* \in S_{X^*}, x^*(x) = 1\}
$
is called the \textit{numerical radius} of $T$. In that study, the authors showed that the extreme case $n(X) = 1$ occurs for a large class of significant spaces, including all real or complex $L$-spaces and $M$-spaces.

In the 1970s, F. F. Bonsall and J. Duncan \cite{BD71, BD73} developed an extensive theory of the numerical range and radius in Banach spaces, extending classical results. They broadened the concept to bounded linear operators on normed spaces and to elements of normed algebras, thereby moving beyond the scope of finite-dimensional matrix theory. These contributions advanced the understanding of the numerical index in more specific Banach spaces, as evidenced in \cite{BKMW07} and \cite{MP00}.

More recently, there was a growing interest in further exploring the numerical range and radius of homogeneous polynomials on Banach spaces. S. G. Kim et al. \cite{CGKM06} introduced the concept of the \textit{polynomial numerical index of order $k$}. Given a Banach space $X$ and $k \in \mathbb{N}$, the polynomial numerical index of order $k$ of $X$ is defined as
\[
n^{(k)}(X) = \inf\{v(P) : P \in \mathcal{P}(^k X, X), \|P\| = 1\},
\]
where
\[
v(P) = \sup\{|x^*(P(x))| : x \in S_X, x^* \in S_{X^*}, x^*(x) = 1\}
\]
is the \textit{numerical radius} of the polynomial \( P \). Further results on the polynomial numerical index are discussed in \cite{GM14}.

The concepts of numerical range, numerical radius, and numerical index, defined with respect to an operator \( G \in \mathcal{L}(X, Y) \) rather than the identity operator, were first introduced by M. A. Ardalani \cite{A14}. In this framework, given \( T, G \in \mathcal{L}(X, Y) \) with \( \|G\| = 1 \), the \emph{numerical range of \( T \) with respect to \( G \)}, denoted by \( V_G(T) \), is defined as
\[
V_G(T) = \bigcap_{\delta >0} \left\{ x^*(Tx) : x \in S_X,\, x^* \in S_{Y^*},\, \operatorname{Re} x^*(Gx) > 1 - \varepsilon \right\}.
\]
The \emph{numerical radius of \( T \) with respect to \( G \)}, denoted by \( v_G(T) \), is then given by
\[
v_G(T) = \sup\{ |t| : t \in V_G(T) \}.
\]
Finally, the \emph{numerical index of the pair \( (X, Y) \) with respect to \( G \)}, denoted by \( n_G(X, Y) \), is defined as
\[
n_G(X,Y) = \inf\{ v_G(T) : T \in \mathcal{L}(X,Y),\ \|T\| = 1 \}.
\]

In this article, we introduce the notion of the polynomial numerical index of a pair \((X, Y)\) with respect to a norm-one polynomial \(Q \in \mathcal{P}(X, Y)\). We develop the concepts of the intrinsic numerical range and the approximated spatial numerical range associated with homogeneous polynomials, establishing several structural properties and characterizations. A connection between these ranges is described via convex hull operations. Moreover, we present various properties of the introduced index. Finally, we study the stability of the polynomial numerical index with respect to a fixed polynomial under $c_0$-, $\ell_\infty$-, and $\ell_1$-sums.

\section{Definitions and preliminaries results}

It is possible to define numerical ranges of functions with respect to a fixed function, as it is done in \cite{M16}. In this paper, we consider the cases of numerical ranges of homogeneous polynomials with respect to a fixed homogeneous polynomial. We now present the definitions of these numerical ranges in the format adopted throughout this work.

The first type of numerical range we introduce has appeared under different names and notations since the 1960s \cite{M16}.

\begin{definition}
     Let $X, Y$ be Banach spaces, and let $Q \in \mathcal{P}(^k X, Y)$ with $\|Q\|=1$.
     For every $P \in \mathcal{P}(^k X, Y)$, the \textit{intrinsic numerical range of $P$ with respect to $Q$} is the set
    $$V(\mathcal{P}(^k X, Y), Q, P) := \{\phi(P) : \phi \in \mathcal{P}(^k X, Y)^*, \|\phi\| = \phi(Q) = 1\}.$$
\end{definition}

In 2014, Ardalani \cite{A14} introduced the concept of approximated spatial numerical range for operators. This concept was later extended to bounded functions by Mart\'in \cite{M16} in a natural way. Here, we present the formulation for homogeneous polynomials.

\begin{definition}
    Let $X, Y$ be Banach spaces, and let $Q \in \mathcal{P}(^k X, Y)$ with $\|Q\|=1$.
    For every $P \in \mathcal{P}(^k X, Y)$, the \textit{approximated spatial numerical range of $P$ with respect to $Q$} is the set
    $$V_Q(P) := \bigcap_{\delta > 0} \overline{\{y^*(P(x)) : y^* \in S_{Y^*}, \ x \in S_X, \ \re y^*(Q(x))>1-\delta\}}.$$
\end{definition}

The relationship between the intrinsic and the approximated spatial numerical ranges of homogeneous polynomials with respect to a fixed homogeneous polynomial is presented in the following result.

\begin{proposition}\label{Prop1}
    Let $X, Y$ be Banach spaces, and let $Q \in \mathcal{P}(^k X, Y)$ with $\|Q\|=1$.
    For every $P \in \mathcal{P}(^k X, Y)$, we have
    $$V(\mathcal{P}(^k X, Y), Q, P) = \conv V_Q(P).$$
\end{proposition}

\begin{proof}
    Note that $\mathcal{P}(^k X, Y)$ is a subspace of $\ell_\infty(S_X, Y)$ that contains both $Q$ and $P$. Thus, by the Hahn–Banach theorem,
    \begin{align*}
    \{\phi(P) : \phi \in \mathcal{P}(^k X, Y)^*, & \ \|\phi\| = \phi(Q) = 1\}\\
    &= \{\varphi(P) : \varphi \in \ell_\infty(S_X, Y)^*, \|\varphi\| = \varphi(Q) = 1\}.
    \end{align*}
    Consider $\Gamma = S_X$, $g = Q$, and $f=P$. Then 
    \begin{align*}
        \bigcap_{\varepsilon > 0} & \overline{\{y^*(f(t)) : y^* \in S_{Y^*}, \ t \in \Gamma, \ \re y^*(g(t))>1-\varepsilon\}} \\
        & \hspace{1.8cm} = \bigcap_{\delta > 0} \overline{\{y^*(P(x)) : y^* \in S_{Y^*}, \ x \in S_X, \ \re y^*(Q(x))>1-\delta\}}.
    \end{align*}
    Therefore, by \cite[Theorem 2.1]{M16}, we conclude that
    $$V(\mathcal{P}(^k X, Y), Q, P) = \conv V_Q(P).$$
\end{proof}

Next, we express the intrinsic numerical range of homogeneous polynomials with respect to a fixed homogeneous polynomial as the convex hull of an alternative intersection, inspired by \cite[Lemma 3.4]{KMMPQ20}.

\begin{proposition}\label{Lemma3.4}
Let $X, Y$ be Banach spaces. Suppose that $A \subseteq B_X$ and $B \subseteq B_{Y^*}$ satisfy
$$
\|p\| = \sup_{x \in A} \re p(x)
$$
for all $p \in \mathcal{P}(^k X)$, and
$$
\cconv^{w^*}(B) = B_{Y^*}.
$$
Then, for every $P, Q \in \mathcal{P}(^k X, Y)$ with $\|Q\|=1$, we have
\begin{align*}
V(P(^k X, Y), & \ Q, P) \\
& = \conv \bigcap_{\delta > 0} \overline{\{ y^*(P(x)) : y^* \in B, \ x \in A, \ \re y^*(Q(x)) > 1 - \delta \}}.
\end{align*}
\end{proposition}

\begin{proof}
Fix $Q \in \mathcal{P}(^k X, Y)$ with $\|Q\|=1$. For \( x \in A \) and \( y^* \in B \), consider the functional \( x \otimes y^* \in \mathcal{P}(^k X,Y)^* \) given by
$$
x \otimes y^*(P) = y^*(P(x)).
$$
Define
$$
A \otimes B = \{ x \otimes y^* : x \in A, \, y^* \in B \}.
$$
Then \( B_{\mathcal{P}(^k X, Y)^*} = \cconv^{w^*} (A \otimes B) \). Indeed, for any \( P \in \mathcal{P}(^k X, Y) \), we have
\begin{align*}
\|P\| &= \sup_{x \in B_X} \|P(x)\|= \sup_{x \in B_X} \sup_{y^* \in B} | y^*(P(x)) | \\
 &= \sup_{y^* \in B} \sup_{x \in B_X} | y^*(P(x)) | = \sup_{y^* \in B} \| y^* \circ P \| \\ 
&= \sup_{y^* \in B} \sup_{x \in A} \re \, (y^* \circ P)(x) = \sup_{x \otimes y^* \in A \otimes B} \re x \otimes y^* (P),
\end{align*}
where the second and fifth equalities follow from the assumptions $\cconv^{w^*}(B)$ $= B_{Y^*}$ and $\|p\| = \sup_{x \in A} \re p(x)$ for all $p \in \mathcal{P}(^k X)$, respectively. Thus, $A \otimes B$ is norming for $\mathcal{P}(^kX, Y)$. Hence
$$
B_{\mathcal{P}(^k X, Y)^*} = \cconv^{w^*} (A \otimes B).
$$
The result follows from \cite[Proposition 2.14]{KMMPQ20}.
\end{proof}

\begin{remark}
    Let \( X \) be a complex Banach space, let \( p \in \mathcal{P}(^kX) \), and let \( A = S_X\). It is not difficult to see that
$$
\| p \| = \sup_{x \in A} \re p(x).
$$ Therefore, Proposition \ref{Prop1} can be viewed as a consequence of Proposition \ref{Lemma3.4} in the complex case.
\end{remark}

The next proposition provides a simplified characterization of the approximated spatial numerical range of homogeneous polynomials with respect to a fixed homogeneous polynomial when the domain is finite-dimen- sional.

\begin{proposition}\label{Prop3.4} 
Let \( X \) be a finite-dimensional Banach space, let \( Y \) be an arbitrary Banach space, and let \( P, Q \in \mathcal{P}(^{k}X, Y) \) with $\|Q\|=1$. Then
$$
V_Q(P) = \{ y^*(P(x)) : x \in S_X, \ y^* \in S_{Y^*}, \ y^*(Q(x)) = 1 \}.
$$
\end{proposition}

\begin{proof} Since \( S_X \) is compact and \( Q \) is a norm-one continuous function, the result follows directly from \cite[Proposition 3.4]{M16}.
\end{proof}

Now, for $P, Q \in \mathcal{P}(^k X, Y)$ with $\|Q\|=1$, consider
$$v(\mathcal{P}(^k X, Y), Q, P) := \sup \{|\lambda| : \lambda \in V(\mathcal{P}(^k X, Y), Q, P)\}$$
and
$$v_Q(P) := \sup \{|\lambda| : \lambda \in V_Q(P)\}.$$ 
By Proposition \ref{Prop1}, these two numbers coincide. That is, the following corollary holds.

\begin{corollary}
    Let $X, Y$ be Banach spaces, and let $Q \in \mathcal{P}(^k X, Y)$ with $\|Q\|=1$.
    Then
    $$v(\mathcal{P}(^k X, Y), Q, P) = v_Q(P)$$
    for every $P\in \mathcal{P}(^k X, Y)$.
    \end{corollary}

The number $v_Q(P)$ is called the \textit{numerical radius of $P$ with respect to $Q$}.

Next, given a norm-one polynomial $Q \in \mathcal{P}(^k X, Y)$ and $\delta > 0$, we define
$$v_{Q,\delta}(P) := \sup \{|y^*(P(x))| : y^* \in S_{Y^*}, \ x \in S_X, \ \re y^*(Q(x))>1-\delta\}$$
for every $P \in \mathcal{P}(^k X, Y)$.

\begin{proposition}\label{p.20}
    Let $X, Y$ be Banach spaces, and let $Q \in \mathcal{P}(^k X, Y)$ with $\|Q\|=1$.
    Then
    $$v_Q(P) = \inf_{\delta > 0} v_{Q,\delta}(P)$$
    for every $P\in \mathcal{P}(^k X, Y)$.
\end{proposition}

\begin{proof}
    Note that
    \begin{align*}
        V_Q(P) & = \bigcap_{\varepsilon > 0} \overline{\{y^*(P(x)) : y^* \in S_{Y^*}, \ x \in S_X, \ \re y^*(Q(x))>1-\varepsilon\}}\\
        & \subseteq  \overline{\{y^*(P(x)) : y^* \in S_{Y^*}, \ x \in S_X, \ \re y^*(Q(x))>1-\delta\}}
    \end{align*}
    for every $\delta > 0$. Thus, we have 
    \begin{align*}
        v_Q & (P) = \sup \{|\lambda| : \lambda \in V_Q(P)\}\\
        & \leq \sup \{|\lambda| : \lambda \in \overline{\{y^*(P(x)) : y^* \in S_{Y^*}, \ x \in S_X, \ \re y^*(Q(x))>1-\delta\}}\}\\
        & = \sup \{|\lambda| : \lambda \in \{y^*(P(x)) : y^* \in S_{Y^*}, \ x \in S_X, \ \re y^*(Q(x))>1-\delta\}\},
    \end{align*}
    for all $\delta > 0$.
    Hence,
    $$v_Q(P) \leq \inf_{\delta > 0} v_{Q,\delta}(P).$$
    Let $I = \inf_{\delta > 0} v_{Q,\delta}(P)$. We shall prove that $v_Q(P) = I$. Since $I$ is an upper bound of $\{|\lambda| : \lambda \in V_Q(P)\}$, it is suffices to show that, for every $\alpha > 0$, there exists $\lambda \in V_Q(P)$ such that $I - \alpha \leq |\lambda|$. Fix $\alpha>0$. Observe that $I - \alpha < v_{Q,\delta}(P)$ for every $\delta > 0$. In particular, for $\delta_n = \frac{1}{n}$, we obtain $I - \alpha < v_{Q,\delta_n}(P)$. So, there is $\lambda_n \in \{y^*(P(x)) : y^* \in S_{Y^*}, \ x \in S_X, \ \re y^*(Q(x))>1-\delta_n\}$ such that $I - \alpha < |\lambda_n|$. Now, note that $|\lambda_n| \leq \|P\|$, because
    $$|y^*(P(x))| \leq \|y^*\|\|P\|\|x\|^k \leq \|P\|$$
    for every $y^* \in S_{Y^*}$ and $x \in S_X$. Hence, $(\lambda_n)$ is a bounded sequence, and admits a convergent subsequence $(\lambda_{n_j} ) \to \lambda \in \conjK$. Let us verify that $\lambda \in V_Q(P)$. Fix $\varepsilon > 0$. Take $n_0 \in \conjN$ such that $\frac{1}{n_0} < \varepsilon$. Thus, if $n \geq n_0$, then
    \begin{align*}
        \lambda_n & \in \{y^*(P(x)) : y^* \in S_{Y^*}, \ x \in S_X, \ \re y^*(Q(x))>1-\delta_n\}\\
        & \subseteq \{y^*(P(x)) : y^* \in S_{Y^*}, \ x \in S_X, \ \re y^*(Q(x))>1-\delta_{n_0}\}\\
        & \subseteq \{y^*(P(x)) : y^* \in S_{Y^*}, \ x \in S_X, \ \re y^*(Q(x))>1-\varepsilon\}.
    \end{align*}
    Therefore, $\lambda \in \overline{\{y^*(P(x)) : y^* \in S_{Y^*}, \ x \in S_X, \ \re y^*(Q(x))>1-\varepsilon\}}$. Since $\varepsilon$ is arbitrary, $\lambda \in V_Q(P)$. Being $I - \alpha < |\lambda_n|$ for every $n$, we obtain $I - \alpha \leq |\lambda|$. This completes the proof.
\end{proof}

We now formally define the polynomial numerical index of a pair of Banach spaces with respect to a norm-one homogeneous polynomial.

\begin{definition}
    Let $X, Y$ be Banach spaces, and let $Q \in \mathcal{P}(^k X, Y)$ with $\|Q\|=1$. The \textit{polynomial numerical index} of $(X, Y)$ with respect to $Q$ is the number
    $$n_Q^{(k)}(X,Y)\coloneqq \inf\{v_Q(P):P\in \mathcal{P}(^k X,Y), \|P\|=1\}.$$
\end{definition}

When \( k = 1 \), this definition recovers the concept of the numerical index of \( (X, Y) \) with respect to a norm-one operator, as introduced in \cite{A14}.

We conclude this section by showing that $n_Q^{(k)}(\conjK, \conjK) = 1$ for every positive integer $k$ and every $Q \in S_{\mathcal{P}(^k \conjK)}$.

\begin{proposition}\label{PropK}
   For every positive integer $k$, we have $n_Q^{(k)}(\conjK, \conjK) = 1$.
\end{proposition}

\begin{proof}
Fix $Q \in S_{\mathcal{P}(^k \conjK)}$. By definition,
$$
n^{(k)}_Q(\conjK, \conjK) = \inf \{ v_Q(P) : P \in S_{\mathcal{P}(^k\conjK)} \}.
$$
According to Proposition \ref{Prop3.4}, for every $P \in \mathcal{P}(^k \conjK)$,
$$
v_Q(P) = \sup \left\{ |y^*(P(x))| : x \in S_{\conjK}, y^* \in S_{\conjK^*}, y^*(Q(x)) = 1 \right\}.
$$ 

Consider $P \in S_{\mathcal{P}(^k \conjK)}$ and \( y^* \in S_{\conjK^*} \). Thus, we can write $P(x) = a x^k$, $Q(x)=bx^k$ and \( y^*(x) = c x \), where $|a|=|b|=|c| = 1$. If $y^*(Q(x)) = 1$, then $cbx^k = 1$, that is, $x^k = \frac{1}{cb}$. In this case,
$$
|y^*(P(x))| = |c\cdot a x^k| = \left|\frac{ca}{cb} \right| = 1.
$$

Hence, \(v_Q(P) = 1 \) for all $P \in S_{ \mathcal{P}(^k\conjK)} $.
Therefore, \( n^{(k)}_Q(\conjK, \conjK) = 1 \).
\end{proof}

\section{Properties of $n_Q^{(k)}(X,Y)$}

In this section, we present several properties of the polynomial numerical index of a pair of Banach spaces with respect to a norm-one homogeneous polynomial. All results are inspired by \cite[Section 3]{KMMPQ20}. Here, $\conjT$ stands for the unit sphere of the base field $\conjK$.

The first result follows directly from \cite[Lemma 2.2]{KMMPQ20}.

\begin{proposition}\label{Lemma3.2}
Let $X, Y$ be Banach spaces, and let $Q \in \mathcal{P}(^k X, Y)$ with $\|Q\|=1$.
Then
$$v_Q(P)=\max_{\theta \in \mathbb{T}}\lim_{\alpha \rightarrow 0^+}\frac{\|Q+\alpha \theta P\|-1}{\alpha}=\lim_{\alpha \rightarrow 0^+}\max_{\theta \in \conjT}\frac{\|Q+\alpha \theta P\|-1}{\alpha}$$
for all $P\in \mathcal{P}(^k X,Y)$. 
\end{proposition}

The next proposition provides a characterization of the polynomial numerical index with respect to a norm-one homogeneous polynomial in terms of the norm of the space $\mathcal{P}(^k X,Y)$. This result is a particular case of \cite[Proposition 2.5]{KMMPQ20}.

\begin{proposition}\label{Prop3.3}
Let $X, Y$ be Banach spaces, let $Q \in \mathcal{P}(^k X, Y)$ with $\|Q\|=1$, and let $0<\lambda \leq 1$. Then the following statements are equivalent:
\begin{enumerate}
    \item[i)] $n_Q^{(k)}(X,Y)\geq \lambda$.
    \item[ii)] $\max_{\theta\in\mathbb{T}}\|Q+\theta P\|\geq 1+\lambda\|P\|$ for every $P\in \mathcal{P}(^k X,Y)$.
\end{enumerate}
\end{proposition}

When $\lambda = 1$ in the above proposition, we recover the notion of spear vector in the space $\mathcal{P}(^k X,Y)$. A norm-one polynomial $Q \in \mathcal{P}(^k X,Y)$ is called a \textit{spear vector} of $\mathcal{P}(^k X,Y)$ if
$$\max_{\theta\in\mathbb{T}}\|Q+\theta P\| = 1 + \|P\|$$ 
for every $P\in \mathcal{P}(^k X,Y)$. The concept of spear vector of a Banach space was introduced in \cite{A14} and studied extensively in \cite{KMMP18}. In this context, Proposition \ref{Prop3.3} shows that $Q$ is a spear vector of $\mathcal{P}(^k X,Y)$ if and only if $n_Q^{(k)}(X,Y) = 1$.

To prove the following result, we require some properties of the adjoint of a polynomial. Given $P \in \mathcal{P}(^k X,Y)$, the \textit{adjoint of} $P$, denoted by $P^*$, is the bounded linear operator $P^*: Y^* \rightarrow \mathcal{P}(^k X)$ defined by $P^*(y^*) = y^* \circ P$. It is straightforward to verify that:
\begin{enumerate}
\item[i)] $\|P^*\| = \|P\|$ for all $P \in \mathcal{P}(^k X,Y)$;
\item[ii)] if $P, Q \in \mathcal{P}(^k X,Y)$ and $\alpha \in \conjK$, then $(Q+\alpha P)^* = Q^*+\alpha P^*$.
\end{enumerate}

\begin{proposition}
Let $X, Y$ be Banach spaces, and let $Q \in \mathcal{P}(^k X, Y)$ with $\|Q\|=1$. Then
$$
n_{Q^*} \left( Y^*, \mathcal{P}(^k X) \right) \leq n_{Q}^{(k)}(X, Y).
$$
\end{proposition} 

\begin{proof}
Let \(P \in \mathcal{P}(^k X, Y)\) with \(\|P\| = 1\). Hence, \(\|Q^*\| = \|P^*\| = 1\). By Proposition \ref{Lemma3.2},
\begin{align*}
    v_{Q}(P) & = \lim_{\alpha \to 0^+} \max_{\theta \in \conjT }\frac{\|Q + \alpha \theta P\|-1}{\alpha}\\ 
    & = \lim_{\alpha \to 0^+} \max_{\theta \in\conjT }\frac{\|Q^* + \alpha \theta P^*\|-1}{\alpha}\\ 
    & = v_{Q^*}(P^*),
\end{align*}
where the last equality follows from \cite[Lemma 3.2]{KMMPQ20}. Thus,
\begin{align*}
    n_{Q^*} (Y^*, \mathcal{P}(^k X)) &= \inf \{ v_{Q^*}(T) : T \in S_{\mathcal{L}(Y^*, \mathcal{P}(^k X))} \}\\ & \leq \inf \{ v_{Q^*}(P^*) : P \in S_{\mathcal{P}(^k X, Y)} \}
    \\ & = \inf \{ v_{Q}(P) : P \in S_{\mathcal{P}(^k X, Y)} \}
    \\ & = n_{Q}^{(k)}(X, Y).
\end{align*}
\end{proof} 

Next, we show that the polynomial numerical index with respect to a norm-one homogeneous polynomial can be controlled by the numerical radius of the operators on the range space. For this, we require the following lemma.

\begin{lemma}\label{Lemma3.7}
    Let $X, Y$ be Banach spaces, and let $Q \in \mathcal{P}(^k X, Y)$ with $\|Q\|=1$. Then
    $$v_{Q}(T \circ Q) \leq v(T)$$
    for every $T \in \mathcal{L}(Y)$.
\end{lemma}

\begin{proof}
    Consider the operator $\Psi\colon \mathcal{L}(Y) \rightarrow \mathcal{P}(^kX,Y)$ defined by $T \mapsto T \circ Q$.
   Since $\|\Psi\|=\|\Psi(\id)\|=\|\id\|=1$, the result follows from \cite[Lemma 2.4]{KMMPQ20}.
\end{proof}

As a consequence, we obtain the next result.

\begin{corollary}\label{p.23}
Let $X, Y$ be Banach spaces, and let $Q \in \mathcal{P}(^k X, Y)$ with $\|Q\|=1$. Then
\begin{align*}
       n_Q^{(k)}(X,Y) &\leq \inf \left\{ \frac{v(T)}{\|T \circ Q\|} : T \in \mathcal{L}(Y), T \circ Q \neq 0 \right\}\\
       & \leq \sup_{\varepsilon > 0} \inf \left\{ v(T) : T \in \mathcal{L}(Y), \|T \circ Q\| > 1 - \varepsilon \right\}.
   \end{align*}
\end{corollary}

\begin{proof}
    By the previous lemma,
    $v_Q(T\circ Q)\leq v(T)$ for every $T \in \mathcal{L}(Y)$. Then,
    \begin{align*}
        n_Q^{(k)}(X,Y)&= \inf\{v_Q(P):P\in \mathcal{P}(^k X,Y), \|P\|=1\}\\
        &\leq \inf \left\{v_Q\left(\frac{T\circ Q}{\|T\circ Q\|}\right): T\in \mathcal{L}(Y), T\circ Q\neq 0 \right\}\\ 
        &= \inf \left\{\frac{v_Q\left(T\circ Q\right)}{\|T\circ Q\|}: T\in \mathcal{L}(Y), T\circ Q\neq 0 \right\}\\ 
        &{\leq} \inf \left\{\frac{v\left(T\right)}{\|T\circ Q\|}: T\in \mathcal{L}(Y), T\circ Q\neq 0 \right\}.
    \end{align*}
    
    Now consider the other inequality. Suppose that
       \begin{align*}
       \sup_{\varepsilon > 0} \inf \{ v(T) : T & \in \mathcal{L}(Y), \|T \circ Q\| > 1 - \varepsilon \} \\
       & < \inf \left\{ \frac{v(T)}{\|T \circ Q\|} : T \in \mathcal{L}(Y), T \circ Q \neq 0 \right\}.
       \end{align*}
       Thus, there exists $a \in \conjR$ such that
       \begin{align*}
       \sup_{\varepsilon > 0} \inf \{ v(T) : T & \in \mathcal{L}(Y), \|T \circ Q\| > 1 - \varepsilon \} < a \\
       & < \inf \left\{ \frac{v(T)}{\|T \circ Q\|} : T \in \mathcal{L}(Y), T \circ Q \neq 0 \right\}.
       \end{align*}
       Note that if $\varepsilon\geq 1$, then
       $$\inf \left\{ v(T) : T \in \mathcal{L}(Y), \|T \circ Q\| > 1 - \varepsilon \right\}=0.$$
       So we take $0< \varepsilon <1$. Consider $(T_n)\subseteq \mathcal{L}(Y)$ such that $\|T_n\circ Q\|>1-\varepsilon$ and
       $$v(T_n)\longrightarrow\inf\left\{ v(T) : T \in \mathcal{L}(Y), \|T \circ Q\| > 1 - \varepsilon \right\}$$
       when $n\rightarrow \infty$. Hence, there exists $n_0\in\conjN$ such that 
       $$v(T_n)<a<\frac{v(T_n)}{\|T_n\circ Q\|}<\frac{v(T_n)}{1-\varepsilon}$$
       for every $n\geq n_0$.
       Letting $n\rightarrow\infty$, we obtain 
       \begin{align*}
       \inf \{ v(T) : T \in \mathcal{L}(Y), \|T \circ Q\| & > 1 - \varepsilon \} < a \\
       & \leq \frac{\inf \{v(T): T \in \mathcal{L}(Y), \|T\circ Q\|>1-\varepsilon\}}{1-\varepsilon}.
       \end{align*}
       Since the function
       \[
       \varepsilon \mapsto \inf \left\{ v(T) : T \in \mathcal{L}(Y), \|T \circ Q\| > 1 - \varepsilon \right\}
       \]
       is non-decreasing as \( \varepsilon \to 0^+ \), we have
       \begin{align*}
       \sup_{\varepsilon>0}\inf \{ v(T) : T \in \mathcal{L}(Y), & \|T \circ Q\| > 1 - \varepsilon \} < a \\
       & \leq \sup_{\varepsilon>0}\inf\left\{ v(T) : T \in \mathcal{L}(Y), \|T \circ Q\| > 1 - \varepsilon \right\},
       \end{align*}
       which is a contradiction.
\end{proof}

The previous corollary immediately yields the following result.

\begin{corollary}\label{Lemma3.8}
    Let $X, Y$ be Banach spaces, let $Q \in \mathcal{P}(^k X, Y)$ with $\|Q\|=1$, and let $0\leq \alpha\leq 1$. If, for every $\varepsilon > 0$, there exists $T_\varepsilon \in \mathcal{L}(Y)$ such that $v(T_\varepsilon) \leq \alpha$ and $\|T_\varepsilon \circ Q \| > 1 - \varepsilon$, then
    $$n_Q^{(k)}(X,Y)\leq \alpha.$$
\end{corollary}

This result leads to some significant consequences.

\begin{proposition}
Let $X, Y$ be Banach spaces, let $Q \in \mathcal{P}(^k X, Y)$ with $\|Q\|=1$, and let $0 \leq \alpha \leq 1$. 
\begin{enumerate}
\item[i)] Consider $\mathcal{B}(\alpha) = \{T \in \mathcal{L}(Y) : \|T\| = 1, \, v(T) \leq \alpha \}$. If, for every $\varepsilon > 0$, the set
$$
\mathfrak{C} = \bigcup_{T \in \mathcal{B}(\alpha)} \{ y \in S_Y : \|T(y)\| > 1 - \varepsilon \}
$$  
is dense in $S_Y$, \textit{then} $n_Q^{(k)}(X,Y) \leq \alpha$.
\item[ii)] If there exists a surjective isometry $T \in \mathcal{L}(Y)$ with $v(T) \leq \alpha$, \textit{then} $n_Q^{(k)}(X,Y) \leq \alpha$.
\end{enumerate}
\end{proposition}

\begin{proof}
   We begin with the proof of \textit{i)}. Since $\|Q\|=1$, for every $\varepsilon > 0$, we may choose $x \in S_X$ such that $\|Q(x)\| > 1 - \varepsilon/3$.
  As $\frac{Q(x)}{\|Q(x)\|}\in S_Y$, by the hypothesis, there exists $y\in \mathfrak{C}$ such that 
  $$\left\|y - \frac{Q(x)}{\|Q(x)\|}\right\| < \varepsilon/3.$$
  Since $y\in \mathfrak{C}$, there is $T_\varepsilon \in \mathcal{L}(Y)$ with $\|T_\varepsilon\| = 1$, $v(T_\varepsilon) \leq \alpha$, and $\|T_\varepsilon (y)\| > 1 - \varepsilon/3$. Now, observe that
  $$\|y - Q(x)\| \leq \left\|y - \frac{Q(x)}{\|Q(x)\|}\right\|+\left\|\frac{Q(x)}{\|Q(x)\|}-Q(x)\right\|<\frac{\varepsilon}{3}+1-\|Q(x)\| < \frac{2\varepsilon}{3}.$$ 
  Therefore,
$$
\|T_\varepsilon (Q(x))\| \geq \|T_\varepsilon (y)\| - \|T_\varepsilon (y - Q(x))\| > 1 - \frac{\varepsilon}{3} - \frac{2\varepsilon}{3} = 1 - \varepsilon.
$$
Hence, $\|T_\varepsilon\circ Q\|>1-\varepsilon$, and the result follows from Corollary \ref{Lemma3.8}.

For \textit{ii)}, note that if \( T \in \mathcal{L}(Y) \) is a surjective isometry with \( v(T) \leq \alpha \), then \( T \in \mathcal{B}(\alpha) \) and
$$\{y\in S_Y: \|T(y)\|>1-\varepsilon\}=S_Y.$$ 
Thus, $\mathfrak{C}=S_Y$ and the result follows from part \textit{i)}.
\end{proof}

When $\alpha=0$, we obtain a stronger result.

\begin{proposition}\label{Prop3.10}
Let $X, Y$ be Banach spaces, and let $Q \in \mathcal{P}(^k X, Y)$ with $\|Q\|=1$. 
   \begin{enumerate}
       \item[i)] If there exists $T\in\mathcal{L}(Y)$ with $v(T)=0$ and $T\circ Q\neq 0$, then $n^{(k)}_Q(X,Y)=0$.
       \item[ii)] If  $$
\mathfrak{D}=\bigcap_{T \in \mathcal{L}(Y), \, v(T) = 0} \ker T = \{0\},$$
then $n_Q^{(k)}(X,Y)=0$.
       \end{enumerate}
\end{proposition}

    \begin{proof}
        Item \textit{i)} follows directly from Corollary \ref{p.23}. 
        For item \textit{ii)}, since $Q\neq 0$, there exists $x\in X$ such that $Q(x)\neq 0$. Then, $Q(x)\notin \mathfrak{D}$, and so there exists $T\in \mathcal{L}(Y)$, with $v(T)=0$ and $T(Q(x))\neq 0$. Hence, the result follows from item \textit{i)}.
    \end{proof}

To conclude the section, we present the following direct consequence of the previous result.

    \begin{corollary}\label{Cor3.11}
        Let $X, Y$ be Banach spaces. If there is an isometry $J\in \mathcal{L}(Y)$ with $v(J)=0$, then
        $$n_Q^{(k)}(X,Y)=0$$
        for every norm-one $Q\in\mathcal{P}(^kX,Y)$.
    \end{corollary}

\section{Stability results}

In this section, we investigate the stability of the $n_{Q}^{(k)}(X, Y)$ under $c_0$-, $\ell_\infty$-, and $\ell_1$-sums. The proof of the following proposition is based on the the argument presented in \cite[Proposition 6.1]{KMMPQ20}.

\begin{proposition}\label{Prop6.1}
    Let $(X_\lambda)_{\lambda \in \Lambda}$, $(Y_\lambda)_{\lambda \in \Lambda}$ be two families of Banach spaces and let $Q_\lambda \in \mathcal{P}(^k X_\lambda, Y_\lambda)$ be a norm-one polynomial for every $\lambda \in \Lambda$. Consider $E$ as one of the Banach spaces $c_0$, $\ell_\infty$ or $\ell_1$, $X = \left[ \bigoplus_{\lambda \in \Lambda} X_\lambda \right]_E$, and $Y = \left[ \bigoplus_{\lambda \in \Lambda} Y_\lambda \right]_E$. Define the $k$-homogeneous polynomial $Q: X \rightarrow Y$ by
    $$Q[(x_\lambda)_{\lambda \in \Lambda}] = (Q_\lambda(x_\lambda))_{\lambda \in \Lambda}$$
    for every $(x_\lambda)_{\lambda \in \Lambda} \in \left[ \bigoplus_{\lambda \in \Lambda} X_\lambda \right]_E$. Then
    $$n_{Q}^{(k)}(X, Y) \leq \inf_{\lambda} n_{Q_\lambda}^{(k)}(X_\lambda, Y_\lambda).$$
\end{proposition}

\begin{proof}
    Fix $\lambda_0 \in \Lambda$. We will prove that $n_{Q}^{(k)}(X, Y) \leq n_{Q_{\lambda_0}}^{(k)}(X_{\lambda_0}, Y_{\lambda_0})$. Setting $W=\left[ \bigoplus_{\lambda \neq \lambda_0} X_\lambda \right]_E$ and $Z=\left[ \bigoplus_{\lambda \neq \lambda_0} Y_\lambda \right]_E$, we can write $X = X_{\lambda_0} \oplus_\infty W$ and $Y = Y_{\lambda_0} \oplus_\infty Z$ when $E$ is $c_0$ or $\ell_\infty$, and $X = X_{\lambda_0} \oplus_1 W$ and $Y = Y_{\lambda_0} \oplus_1 Z$ when $E$ is $\ell_1$. Given $R \in \mathcal{P}(^k X_{\lambda_0}, Y_{\lambda_0})$ with $\|R\| = 1$, define $P \in \mathcal{P}(^k X, Y)$ by
    $$P (x_{\lambda_0},w) = (R(x_{\lambda_0}), 0)$$
    for $x_{\lambda_0} \in X_{\lambda_0}$ and $w \in W$. Note that $\|P\| = \|R\| = 1$. We claim that $v_{Q_{\lambda_0}}(R) = v_Q(P)$.
    
    Let us first show that $v_{Q_{\lambda_0}}(R) \leq v_Q(P)$. Fix $\delta > 0$, $x_{\lambda_0} \in S_{X_{\lambda_0}}$, and $y_{\lambda_0}^* \in S_{Y_{\lambda_0}^*}$ satisfying $\re y_{\lambda_0}^*(Q_{\lambda_0}(x_{\lambda_0})) > 1 -\delta$. Define $x = (x_{\lambda_0}, 0) \in S_X$ and $y^* = (y_{\lambda_0}^*, 0) \in S_{Y^*}$. Then,
    $$\re y^*(Q(x)) = \re y_{\lambda_0}^*(Q_{\lambda_0}(x_{\lambda_0})) > 1 - \delta$$
    and
    $$|y_{\lambda_0}^*(R(x_{\lambda_0}))| = |y^* (P(x))| \leq v_{Q, \delta}(P).$$
    So, $v_{Q_{\lambda_0}, \delta}(R) \leq v_{Q, \delta}(P)$ and the desired inequality follows by letting $\delta \downarrow 0$.
    
    To obtain the reverse inequality, we may assume $v_Q(P) > 0$. Otherwise, $v_Q(P) = 0$ and the inequality follows directly. Given $\delta > 0$, it is sufficient to prove $v_{Q, \delta}(P) \leq v_{Q_{\lambda_0}, \widehat{\delta}}(R)$, where $\widehat{\delta} = 2\delta/v_Q(P)$. In fact, suppose instead that $v_{Q, \delta}(P) \leq v_{Q_{\lambda_0}, \widehat{\delta}}(R)$ for every $\delta > 0$, and $v_{Q_{\lambda_0}}(R) < v_Q(P)$. Then there exists $\eta > 0$ such that
    $$v_{Q_{\lambda_0}, \eta}(R) < v_{Q, \delta}(P),$$
    for all $\delta > 0$. Taking $\delta = \eta v_Q(P)/ 2 > 0$, we get 
    $$v_{Q_{\lambda_0}, \eta}(R) < v_{Q, \delta}(P) \leq v_{Q_{\lambda_0}, \widehat{\delta}}(R),$$
    where $\widehat{\delta} = 2\delta/v_Q(P) = \eta$, which is a contradiction. Hence, if $v_{Q, \delta}(P) \leq v_{Q_{\lambda_0}, \widehat{\delta}}(R)$ for every $\delta > 0$, then $v_{Q_{\lambda_0}}(R) \geq v_Q(P)$.
    
    Now, let $0 < \varepsilon < v_Q(P)/2$. There exist $x = (x_{\lambda_0}, w) \in S_X$ and $y^* = (y_{\lambda_0}^*, z^*) \in S_{Y^*}$ such that 
    $$|y^*(P(x))| > v_{Q, \delta}(P) - \varepsilon \geq v_Q(P) - \varepsilon > \frac{v_Q(P)}{2}$$
    and
    \begin{align*}
    1 - \delta & < \re y^*(Q(x)) = \re y_{\lambda_0}^*(Q_{\lambda_0}(x_{\lambda_0})) + \re z^*((Q_\lambda(w_\lambda))_{\lambda \neq \lambda_0}) \\
    & \leq \re y_{\lambda_0}^*(Q_{\lambda_0}(x_{\lambda_0})) + \|z^*\| \|w\|^k \leq \re y_{\lambda_0}^*(Q_{\lambda_0}(x_{\lambda_0})) + \|z^*\| \|w\|.
    \end{align*}
    Also,
    $$\|y_{\lambda_0}^*\| \|x_{\lambda_0}\| + \|z^*\| \|w\| \leq \|y^*\| \|x\| = 1.$$
    Thus,
    \begin{align*}
    \|y_{\lambda_0}^*\| \|x_{\lambda_0}\|^k + \|z^*\| \|w\| - \delta & \leq \|y_{\lambda_0}^*\| \|x_{\lambda_0}\| + \|z^*\| \|w\| - \delta \\
    & \leq 1 - \delta < \re y_{\lambda_0}^*(Q_{\lambda_0}(x_{\lambda_0})) + \|z^*\| \|w\|,
    \end{align*}
    implying $\re y_{\lambda_0}^*(Q_{\lambda_0}(x_{\lambda_0})) > \|y_{\lambda_0}^*\| \|x_{\lambda_0}\|^k - \delta$. Since
    $$0 < \frac{v_Q(P)}{2} < |y^*(P(x))| = |y_{\lambda_0}^*(R(x_{\lambda_0}))| \leq \|y_{\lambda_0}^*\| \|x_{\lambda_0}\|^k,$$
    we have
    $$\re \frac{y_{\lambda_0}^*}{\|y_{\lambda_0}^*\|} \left( Q_{\lambda_0} \left(\frac{x_{\lambda_0}}{\|x_{\lambda_0}\|} \right) \right) > 1 - \frac{\delta}{\|y_{\lambda_0}^*\| \|x_{\lambda_0}\|^k} > 1 - \frac{2 \delta}{v_Q(P)} = 1 - \widehat{\delta}.$$
    Consequently,
    \begin{align*}
    v_{Q, \delta}(P) - \varepsilon & < |y^*(P(x))| = |y_{\lambda_0}^*(R(x_{\lambda_0}))| \\
    & \leq \left| \frac{y_{\lambda_0}^*}{\|y_{\lambda_0}^*\|} \left( R \left(\frac{x_{\lambda_0}}{\|x_{\lambda_0}\|} \right) \right) \right| \leq v_{Q_{\lambda_0}, \widehat{\delta}}(R).
    \end{align*}
    Letting $\varepsilon \downarrow 0$, we obtain $v_{Q, \delta}(P) \leq v_{Q_{\lambda_0}, \widehat{\delta}}(R)$.

    In summary, we showed that given $R \in \mathcal{P}(^k X_{\lambda_0}, Y_{\lambda_0})$ with $\|R\|=1$, there is $P \in \mathcal{P}(^k X, Y)$ with $\|P\| = 1$ and $v_Q(P) = v_{Q_{\lambda_0}}(R)$. Therefore,
    $$n_{Q}^{(k)}(X, Y) \leq v_Q(P) = v_{Q_{\lambda_0}}(R)$$
    and taking the infimum over all such \( R \) yields
    $n_{Q}^{(k)}(X, Y) \leq n_{Q_{\lambda_0}}^{(k)}(X_{\lambda_0}, Y_{\lambda_0})$.
\end{proof}

\begin{remark}
The equality in Proposition \ref{Prop6.1} does not hold in general. In fact, the equality holds in the case $k=1$ \cite[Proposition 6.1]{KMMPQ20}, but it fails for $k = 2$ as we will prove.
Consider $\ell_1^2$ as $(\mathbb{K}^2, \| \cdot \|_1)$. Define $Q_1, Q_2: \mathbb{K} \to \mathbb{K}$ by $Q_1(a) = a^2 = Q_2(a)$. Note that $Q_1, Q_2$ are norm-one $2$-homogeneous polynomials and that $n_{Q_1}^{(2)}(\conjK, \conjK) = n_{Q_2}^{(2)}(\conjK, \conjK) = 1$, by Proposition \ref{PropK}. Taking $Q: \ell_1^2 \to \ell_1^2$ as in Proposition \ref{Prop6.1}, that is,
$$Q(x_1,x_2) = (x_1^2, x_2^2),$$
we will verify that 
$$n_Q^{(2)}(\ell_1^2, \ell_1^2) \leq \frac{1}{2} < 1 = \min \left\{ n_{Q_1}^{(2)}(\conjK, \conjK), n_{Q_2}^{(2)}(\conjK, \conjK) \right\}.$$

Let $P \in \mathcal{P}(^2 \ell_1^2, \ell_1^2)$ defined by  
$$
P(x_1,x_2) = \left( \frac{x_1^2}{2} + 2x_1x_2, - \frac{x_2^2}{2} - x_1x_2 \right).
$$
We have that $P$ is a 2-homogeneous polynomial because it is associated with the 2-linear mapping $A: \ell_1^2 \times \ell_1^2 \rightarrow \ell_1^2$ given by  
$$
    A((x_1, x_2), (y_1, y_2)) = \left( \frac{x_1y_1}{2} + 2x_1y_2, -\frac{x_2y_2}{2} - x_1y_2 \right).
$$
Let us prove that $\|P\| = 1$ and $v_Q(P) = \frac{1}{2}$. Observe that  
\begin{align*}
    \|P\| & = \sup \left\{ \left\|\left( \frac{x_1^2}{2} + 2x_1x_2, - \frac{x_2^2}{2} - x_1x_2 \right)\right\| \colon (x_1, x_2) \in B_{\ell_1^2} \right\}\\
    & = \sup \left\{ \left| \frac{x_1^2}{2} + 2x_1x_2 \right|+\left| - \frac{x_2^2}{2} - x_1x_2 \right| \colon (x_1, x_2) \in B_{\ell_1^2} \right\}.
\end{align*}
For all $(x_1, x_2) \in B_{\ell_1^2}$, 
\begin{align*}
    \left| \frac{x_1^2}{2} + 2x_1x_2 \right| + \left| - \frac{x_2^2}{2} - x_1x_2 \right| & \leq \frac{1}{2} (|x_1| + |x_2|)^2 + 2|x_1||x_2|\\
    & \leq \frac{1}{2} + 2|x_1||x_2|\\
    & \leq \frac{1}{2} +2(1-|x_2|)|x_2|\\ 
    & =\frac{1}{2}-2\left[\left(|x_2|-\frac{1}{2}\right)^2-\frac{1}{4}\right]\\ 
    & = 1 - 2 \left( |x_2| - \frac{1}{2} \right)^2 \leq 1.
\end{align*}
Thus,
$$
\|P\| = \sup \left\{ \left| \frac{x_1^2}{2} + 2x_1x_2 \right| + \left| - \frac{x_2^2}{2} - x_1x_2 \right| \colon (x_1, x_2) \in B_{\ell_1^2} \right\} \leq 1.
$$
On the other hand, consider $x = \left( \frac{1}{2}, \frac{1}{2} \right) \in B_{\ell_1^2}$. Hence, $\|P(x)\| = 1$. Therefore, $\|P\| = 1$.

Now, we will show that $v_Q(P) = \frac{1}{2}$. Let $(x, y^*)$ be such that $x = (x_1, x_2) \in S_{\ell_1^2}, y^* \in S_{(\ell_1^2)^*}$, and $y^*(Q(x)) = 1$. Since $(\ell_1^2)^* = \ell_\infty^2$ and $y^* \in S_{(\ell_1^2)^*}$, there exists $b = (b_1, b_2) \in S_{\ell_\infty^2}$ such that $y^*(y_1, y_2) = y_1 b_1 + y_2 b_2$ for all $(y_1, y_2) \in \ell_1^2$. We know that  
$$
y^*(Q(x)) = \sum_{j=1}^{2} x_j^2 b_j = 1.
$$
Hence,  
$$
\sum_{j=1}^{2} |x_j| = 1 = \left|\sum_{j=1}^{2} x_j^2 b_j\right|
\leq \sum_{j=1}^{2} |x_j|^2 |b_j| 
\leq \sum_{j=1}^{2} |x_j|^2 
\leq \sum_{j=1}^{2} |x_j|,
$$
that implies
$$
|x_1|+|x_2| = |x_1|^2+|x_2|^2.
$$
Since $|x_1|+|x_2| = 1$, we have only two possibilities. The first case is $|x_1| = 1$ and $|x_2|=0$. And the second is $|x_1| = 0$ and $|x_2|=1$. In both cases, $|y^*(P(x))| = \frac{1}{2}$. Consequently, $v_Q(P) = \frac{1}{2}$ by Proposition \ref{Prop3.4}. 

Therefore, $n_Q^{(2)}{(\ell_1^2, \ell_1^2)} \leq \frac{1}{2}$.
\end{remark}

We observe that the proof of Proposition \ref{Prop6.1} also applies to \(\ell_p\)-sums.

\begin{proposition}\label{Cor6.3}
    Let $(X_\lambda)_{\lambda \in \Lambda}$, $(Y_\lambda)_{\lambda \in \Lambda}$ be two families of Banach spaces, let $Q_\lambda \in \mathcal{P}(^k X_\lambda, Y_\lambda)$ be a norm-one polynomial for every $\lambda \in \Lambda$, let $1 < p < \infty$, and let $X = \left[ \bigoplus_{\lambda \in \Lambda} X_\lambda \right]_{\ell_p}$ and $Y = \left[ \bigoplus_{\lambda \in \Lambda} Y_\lambda \right]_{\ell_p}$. Define the $k$-homogeneous polynomial $Q: X \rightarrow Y$ by
    $$Q[(x_\lambda)_{\lambda \in \Lambda}] = (Q_\lambda(x_\lambda))_{\lambda \in \Lambda}$$
    for every $(x_\lambda)_{\lambda \in \Lambda} \in \left[ \bigoplus_{\lambda \in \Lambda} X_\lambda \right]_{\ell_p}$. Then
    $$n_{Q}^{(k)}(X, Y) \leq \inf_{\lambda} n_{Q_\lambda}^{(k)}(X_\lambda, Y_\lambda).$$
\end{proposition}

\begin{remark}
The equality in Proposition \ref{Cor6.3} also does not hold in general. 
In fact, consider $1 < p < \infty$, an integer $k \geq p$, and $\ell_p^2$ as $(\mathbb{K}^2, \| \cdot \|_p)$. Define $Q_1, Q_2: \mathbb{K} \to \mathbb{K}$ by $Q_1(a) = a^k = Q_2(a)$. Then, $Q_1, Q_2$ are norm-one $k$-homogeneous polynomials and $n_{Q_1}^{(k)}(\conjK, \conjK) = n_{Q_2}^{(k)}(\conjK, \conjK) = 1$, by Proposition \ref{PropK}. Taking $Q: \ell_p^2 \to \ell_p^2$ as in Proposition \ref{Cor6.3},
we will prove that 
$$n_Q^{(k)}(\ell_p^2, \ell_p^2) = 0 < 1 = \min \left\{ n_{Q_1}^{(k)}(\conjK, \conjK), n_{Q_2}^{(k)}(\conjK, \conjK) \right\}.$$

Given $P \in \mathcal{P}(^k \ell_p^2, \ell_p^2)$ defined by  
$$
P(x_1,x_2) = \left( x_2^k, x_1^k \right),
$$
let us show that $\|P\| = 1$ and $v_Q(P) = 0$. First, observe that
\begin{align*}
    \|P\| & = \sup \left\{ \left\|\left( x_2^k, x_1^k \right)\right\| \colon (x_1, x_2) \in B_{\ell_p^2} \right\}\\
    & = \sup \left\{ (|x_1|^{kp} + |x_2|^{kp})^{1/p} \colon (x_1, x_2) \in B_{\ell_p^2} \right\}\\
    & \leq \sup \left\{ (|x_1|^{p} + |x_2|^{p})^{1/p} \colon (x_1, x_2) \in B_{\ell_p^2} \right\} = 1.
\end{align*}
Besides that, for $x = (1,0) \in B_{\ell_p^2}$, we have $\|P(x)\| = 1$. Therefore, $\|P\| = 1$.

Now, consider $(x, y^*)$ such that $x = (x_1, x_2) \in S_{\ell_p^2}, y^* \in S_{(\ell_p^2)^*}$, and $y^*(Q(x)) = 1.$ Since $(\ell_p^2)^* = \ell_q^2$, where $q$ is the conjugate index of $p$, and $y^* \in S_{(\ell_p^2)^*}$, there is $b = (b_1, b_2) \in S_{\ell_q^2}$ such that $y^*(y_1, y_2) = y_1 b_1 + y_2 b_2$ for every $(y_1, y_2) \in \ell_p^2$. Thus,
$$
y^*(Q(x)) = \sum_{j=1}^{2} x_j^k b_j = 1.
$$
Consequently, 
$$
1 = \left|\sum_{j=1}^{2} x_j^k b_j\right|
\leq \sum_{j=1}^{2} |x_j|^k |b_j| 
\leq \sum_{j=1}^{2} |x_j|^p |b_j|
\leq \sum_{j=1}^{2} |x_j|^p = 1,
$$
that implies
$$
\sum_{j=1}^{2} |x_j|^p |b_j| = 1.
$$
Since $|x_1|^p + |x_2|^p = 1$, we have $|b_j|=1$ if $x_j \neq 0$. But this is only possible if \(x_1 = 0\) or \(x_2 = 0\), since otherwise \((b_1, b_2)\) would have both coordinates of modulus 1, contradicting \(\|(b_1, b_2)\|_q = 1\). In addition, if $x_1 = 0$ then $|x_2|=|b_2|=1$ and $b_1=0$. Analogously, if $x_2 = 0$ then $|x_1|=|b_1|=1$ and $b_2=0$. In both cases, $y^*(P(x)) = 0$. So, $v_Q(P) = 0$ by Proposition \ref{Prop3.4}. 

Therefore, $n_Q^{(2)}{(\ell_p^2, \ell_p^2)} = 0$.
\end{remark}

In the following proposition, we prove that the equality holds for the $c_0$- and $\ell_\infty$-sum of complex Banach spaces. The proof is inspired by \cite[Proposition 2.3]{CGMM08} and \cite[Proposition 6.1]{KMMPQ20}.

\begin{proposition}
    Let $(X_\lambda)_{\lambda \in \Lambda}$, $(Y_\lambda)_{\lambda \in \Lambda}$ be two families of complex Banach spaces, and let $Q_\lambda \in \mathcal{P}(^k X_\lambda, Y_\lambda)$ be a norm-one polynomial for every $\lambda \in \Lambda$. Consider $E$ as $c_0$ or $\ell_\infty$, $X = \left[ \bigoplus_{\lambda \in \Lambda} X_\lambda \right]_E$, and $Y = \left[ \bigoplus_{\lambda \in \Lambda} Y_\lambda \right]_E$. Define the $k$-homogeneous polynomial $Q: X \rightarrow Y$ as in Proposition \ref{Prop6.1}. Then
    $$n_{Q}^{(k)}(X, Y) = \inf_{\lambda} n_{Q_\lambda}^{(k)}(X_\lambda, Y_\lambda).$$
\end{proposition}

\begin{proof}
    We already proved in the previous proposition that
    $$n_{Q}^{(k)}(X, Y) \leq \inf_{\lambda} n_{Q_\lambda}^{(k)}(X_\lambda, Y_\lambda).$$
    Thus, it suffices to establish the reverse inequality.
    
    Since $E$ is $c_0$ or $\ell_\infty$, a polynomial $P \in \mathcal{P}(^k X, Y)$ can be seen as a family $(P_\lambda)_{\lambda \in \Lambda}$, where $P_\lambda \in \mathcal{P}(^k X, Y_\lambda)$ for every $\lambda$, and $\|P\| = \sup \{ \|P_\lambda\| : \lambda \in \Lambda \}$. Given $\varepsilon > 0$, there is $\lambda_0 \in \Lambda$ such that $\|P_{\lambda_0}\| > \|P\| - \varepsilon$. Write $X = X_{\lambda_0} \oplus_\infty W$, where $W=\left[ \bigoplus_{\lambda \neq \lambda_0} X_\lambda \right]_E$. Choose $(x_0,w_0) \in S_X$ such that $x_0 \in X_{\lambda_0}$, $w_0 \in W$, and $\|P_{\lambda_0}(x_0,w_0)\| > \|P\| - \varepsilon$. Let us show that we may assume $\|x_0\| = 1$. Indeed, consider $x_1 \in S_{X_{\lambda_0}}$ such that $\|x_0\| x_1 = x_0$, and choose $x_{\lambda_0}^* \in S_{X_{\lambda_0}^*}$ satisfying
    $$|x_{\lambda_0}^*(P_{\lambda_0}(x_0,w_0))| > \|P\| - \varepsilon.$$
    Define the entire function $g : \conjC \rightarrow \conjC$ by
    $$g(z) = x_{\lambda_0}^*(P_{\lambda_0}(zx_1,w_0))$$
    for all $z \in \conjC$. By the Maximum Modulus Theorem, we have
    \begin{align*}
    \|P\| - \varepsilon & < |x_{\lambda_0}^*(P_{\lambda_0}(x_0,w_0))| = |g(\|x_0\|)| \leq |g(z_0)| \\
    & = |x_{\lambda_0}^*(P_{\lambda_0}(z_0 x_1,w_0))| \leq \|P_{\lambda_0}(z_0 x_1,w_0)\|
    \end{align*}
    for some $z_0 \in \conjC$ with $|z_0| = 1$. So, replacing $x_0$ by $z_0x_1$, the claim is done. 
    
    Now fix $x_0^* \in S_{X_{\lambda_0}^*}$ with $x_0^*(x_0) = 1$, and define the polynomial $R \in \mathcal{P}(^k X_{\lambda_0}, Y_{\lambda_0})$ by
    $$R(x) = P_{\lambda_0}(x, x_0^*(x)w_0)$$
    for every $x \in X_{\lambda_0}$. Then,
    $$\|R\| \geq \|R(x_0)\| = \|P_{\lambda_0}(x_0, x_0^*(x_0)w_0)\| = \|P_{\lambda_0}(x_0, w_0)\| > \|P\| - \varepsilon.$$
    Given $\delta > 0$, we claim that $v_{Q_{\lambda_0}, \delta}(R) \leq v_{Q, \delta}(P)$. Indeed, let $u \in S_{X_{\lambda_0}}$ and $v^* \in S_{Y^*_{\lambda_0}}$ be such that $\re v^*(Q_{\lambda_0}(u)) > 1 - \delta$. Take
    $$x = (u, x_0^*(u)w_0) \in S_X, \ \ y^* = (v^*, 0) \in S_{Y^*}.$$
    Observe that
    $$\re y^* (Q(x)) = \re v^*(Q_{\lambda_0}(u)) > 1 - \delta.$$
    Hence,
    $$|v^*(R(u))| = |v^*(P_{\lambda_0}(u, x_0^*(u)w_0))| = |y^*(P(x))| \leq v_{Q, \delta}(P).$$
    Thus, we deduce $v_{Q_{\lambda_0}, \delta}(R) \leq v_{Q, \delta}(P)$. Consequently,
    $$v_Q(P) \geq v_{Q_{\lambda_0}}(R) \geq n_{Q_{\lambda_0}}^{(k)}(X_{\lambda_0}, Y_{\lambda_0}) \|R\| \geq n_{Q_{\lambda_0}}^{(k)}(X_{\lambda_0}, Y_{\lambda_0}) [\|P\| - \varepsilon].$$
    Therefore,
    $$v_Q(P) \geq \inf_{\lambda} n_{Q_\lambda}^{(k)}(X_\lambda, Y_\lambda)\|P\|,$$
    and so, $n_{Q}^{(k)}(X, Y) \geq \inf_{\lambda} n_{Q_\lambda}^{(k)}(X_\lambda, Y_\lambda)$.
\end{proof}

\begin{remark}
The previous proposition does not hold in the real case. Denote $(\mathbb{R}^2, \| \cdot \|_\infty)$ by $\ell^2_\infty$. Define $Q_1, Q_2: \conjR \to \conjR$ by $Q_1(a) = a^3 = Q_2(a)$. 
Thus, $Q_1, Q_2$ are norm-one $3$-homogeneous polynomials, and $n_{Q_1}^{(3)}(\conjR, \conjR) = n_{Q_2}^{(3)}(\conjR, \conjR) = 1$, by Proposition \ref{PropK}. Assuming $Q: \ell_\infty^2 \to \ell_\infty^2$ as in Proposition \ref{Prop6.1}, we will show that
$$n_Q^{(3)}(\ell^2_\infty, \ell^2_\infty) \leq \frac{2}{3\sqrt{3}} < 1 = \min \left\{ n_{Q_1}^{(3)}(\conjR, \conjR), n_{Q_2}^{(3)}(\conjR, \conjR) \right\}.$$

Let $P \in \mathcal{P}(^3\ell^2_\infty; \ell^2_\infty)$ be given by 
$$P(x_1,x_2) = (x_1^2x_2 - x_2^3, 0).$$ We will verify that $\|P\| = 1$ and $v_Q(P) \leq \frac{2}{3\sqrt{3}}$. For every $(x_1,x_2) \in \ell^2_\infty$ with $\|(x_1,x_2)\|_\infty\leq 1$, 
$$    \|P(x_1,x_2)\|_\infty = \|(x_1^2x_2 - x_2^3, 0)\|_\infty = |x_1^2x_2 - x_2^3| \leq |x_1^2 - x_2^2| \leq 1.
$$
Then, $\|P\| \leq 1$. Taking $x=(0,1) \in B_{\ell^2_\infty}$, we obtain $\|P(x)\| = 1$. Hence, $\|P\| = 1$.

Next, let $x = (x_1,x_2) \in S_{\ell_\infty^2}$ and $y^* \in S_{(\ell_\infty^2)^*}$ be such that $y^*(Q(x)) = 1$. Since $(\ell_\infty^2)^* = \ell^2_1$, there exists $b = (b_1, b_2) \in S_{\ell^2_1}$ such that $y^*(y_1, y_2) = y_1b_1+y_2b_2$. Therefore, $$y^*(Q(x)) = \sum_{j=1}^2 x_j^3b_j = 1.$$
We now compute \(|y^*(P(x))|\) by analyzing three separate cases:

\textbf{Case 1:} Suppose that $|x_1| = 1$ and $|x_2| = 1$. Then
$$
P(x) = (x_1^2x_2 - x_2^3, 0) = (x_2 (x_1^2 - x_2^2), 0) = (0,0),
$$
so $|y^*(P(x))| = 0$.

\textbf{Case 2:} Suppose that $|x_1| = 1$ and $|x_2| < 1$. Thus,
$$
\sum_{j=1}^2 |b_j| = 1 = \left| \sum_{j=1}^2 x_j^3b_j \right| \leq |b_1| + |x_2||b_2|,
$$
that implies $|b_1| = 1$ and $|b_2|=0$. Hence,
\begin{align*}
    |y^*(P(x))| = |(x_1^2 x_2 - x_2^3)b_1| = |x_2-x_2^3|.
\end{align*}

\textbf{Case 3:} Suppose that $|x_1| < 1$ and $|x_2| = 1$. Analogously, we obtain \(|b_1| = 0\), \(|b_2| = 1\), and thus, $|y^*(P(x))| = 0$. 

Consequently, by Proposition \ref{Prop3.4},
$$ v_Q(P) = \sup \{|x_2 - x_2^3| : |x_2| < 1\} = \frac{2}{3\sqrt{3}}.$$
Therefore,
$$
n_Q^{(2)} (\ell_\infty^2, \ell^2_\infty) \leq \frac{2}{3\sqrt{3}}.
$$ 
\end{remark}

\section*{Disclosure statement}

The authors report there are no competing interests to declare.


\end{document}